\documentclass[12pt,english]{smfart}

\usepackage{smfthm}
\usepackage{amssymb}
\usepackage[applemac]{inputenc}

\renewcommand{\phi}{\varphi}
\renewcommand{\projlim}{\varprojlim}

\renewcommand{\geq}{\geqslant}
\renewcommand{\leq}{\leqslant} 
\renewcommand{\hat}{\widehat}

\newcommand{\Qp}{\mathbf{Q}_p}
\newcommand{\Fp}{\mathbf{F}_p}
\newcommand{\Fpn}{\mathbf{F}_{p^n}}
\newcommand{\Zp}{\mathbf{Z}_p}
\newcommand{\ZZ}{\mathbf{Z}}
\newcommand{\QQ}{\mathbf{Q}}
\newcommand{\RR}{\mathbf{R}}
\newcommand{\Qpbar}{\overline{\mathbf{Q}}_p}
\newcommand{\Fpbar}{\overline{\mathbf{F}}_p}
\newcommand{\B}{\mathrm{B}}
\newcommand{\K}{\mathrm{K}}
\newcommand{\I}{\mathrm{I}_1}
\newcommand{\Z}{\mathrm{Z}}
\newcommand{\gal}{\mathcal{G}_{\Qp}}
\newcommand{\hal}{\mathcal{H}_{\Qp}}
\newcommand{\weil}{\mathcal{W}_{\Qp}}
\newcommand{\inert}{\mathcal{I}_{\Qp}}
\newcommand{\Gal}{\mathrm{Gal}}
\newcommand{\ind}{\mathrm{ind}}
\newcommand{\pr}{\mathrm{pr}}
\newcommand{\Mat}{\mathrm{Mat}}
\newcommand{\NP}{\mathrm{NP}}
\newcommand{\Frob}{\mathrm{Frob}}
\newcommand{\M}{\mathrm{M}}
\newcommand{\GL}{\mathrm{GL}}
\newcommand{\Id}{\mathrm{Id}}
\newcommand{\val}{\mathrm{val}}

\newcommand{\vr}{\mathrm{val}_r}
\newcommand{\vx}{\mathrm{val}_X}
\newcommand{\efont}{\mathbf{E}}
\newcommand{\dtilde}{\widetilde{\mathrm{D}}}
\newcommand{\dfont}{\mathrm{D}}
\newcommand{\mfont}{\mathrm{M}}
\newcommand{\mtor}{\mathrm{M}_{\mathrm{tor}}}
\newcommand{\nfont}{\mathrm{N}}
\newcommand{\dpar}[1]{(\!( #1 )\!)}
\newcommand{\dcroc}[1]{[\![ #1 ]\!]}
\newcommand{\smat}[1]{\left( \begin{smallmatrix} #1 \end{smallmatrix} \right)}
\newcommand{\pmat}[1]{\begin{pmatrix} #1 \end{pmatrix}}
\newcommand{\disc}{\mathrm{Vec}_{\mathrm{disc}}(E)}
\newcommand{\comp}{\mathrm{Vec}_{\mathrm{comp}}(E)}
\newcommand{\discg}{\mathrm{Vec}^G_{\mathrm{disc}}(E)}
\newcommand{\compg}{\mathrm{Vec}^G_{\mathrm{comp}}(E)}

\author{Laurent Berger}
\address{UMPA, ENS de Lyon \\
UMR 5669 du CNRS \\
Universit\'e de Lyon}
\email{laurent.berger@ens-lyon.fr}
\urladdr{perso.ens-lyon.fr/laurent.berger/}

\author{Mathieu Vienney}
\address{UMPA, ENS de Lyon \\
UMR 5669 du CNRS \\
Universit\'e de Lyon}
\email{mathieu.vienney@ens-lyon.fr}
\urladdr{http://www.umpa.ens-lyon.fr/\~{}mvienney/}

\date{March 2012}

\title{Irreducible modular representations of the Borel subgroup of $\mathrm{GL}_2(\mathbf{Q}_p)$}

\subjclass{11F; 11S; 20C; 20G; 22E; 46A}

\keywords{Modular representation; Borel subgroup; Galois representation; $(\phi,\Gamma)$-module; $(\psi,\Gamma)$-module; $p$-adic Langlands correspondence; linear topology; non-commutative polynomial}

\thanks{This research is partially supported by the ANR grant Th\'eHopaD (Th\'eorie de Hodge $p$-adique et D\'eveloppements) ANR-11-BS01-005}

\begin{document}

\begin{abstract}
Let $E$ be a finite extension of $\Fp$. Using Fontaine's theory of $(\phi,\Gamma)$-modules, Colmez has shown how to attach to any irreducible $E$-linear representation of $\Gal(\Qpbar/\Qp)$ an infinite dimensional smooth irreducible $E$-linear representation of $\B_2(\Qp)$ that has a central character. We prove that every such representation of $\B_2(\Qp)$ arises in this way. 

Our proof extends to algebraically closed fields $E$ of characteristic $p$. In this case, infinite dimensional smooth irreducible $E$-linear representations of $\B_2(\Qp)$ having a central character arise in a similar way from irreducible $E$-linear representations of the Weil group of $\Qp$.
\end{abstract}

\maketitle

\tableofcontents

\setlength{\baselineskip}{18pt}

\section*{Introduction}\label{intro}


This article is inspired by the $p$-adic local Langlands correpondence for $\GL_2(\Qp)$, which is a bijection between some $2$-dimensional representations of $\Gal(\Qpbar/\Qp)$ and some representations of $\GL_2(\Qp)$. Colmez observed that this bijection, whose existence had been conjectured by Breuil, can be constructed using Fontaine's theory of $(\phi,\Gamma)$-modules, in order to obtain representations of $\B_2(\Qp)$, the upper triangular Borel subgroup of $\GL_2(\Qp)$, from $2$-dimensional representations of $\Gal(\Qpbar/\Qp)$. In this article, we determine completely which class of representations of $\B_2(\Qp)$ can be constructed by applying Colmez's method to irreducible mod $p$ representations of $\Gal(\Qpbar/\Qp)$ of any dimension.

Let $E$ be a finite extension of $\Fp$. If $V$ is a finite dimensional $E$-linear representation of $\Gal(\Qpbar/\Qp)$ and if $\chi$ is a smooth character of $\Qp^\times$, then Colmez's functor ``$\projlim_\psi D^\natural(\cdot)$'' allows us to construct a smooth representation $\Omega_\chi(V) = (\projlim_\psi D^\natural(V))^*$ of the group $\B_2(\Qp)$, having $\chi$ as central character. Our first result is the following (see also \cite{MVart}).

\begin{enonce*}{Theorem A}
If $E$ is a finite field, and if $\Pi$ is an infinite dimensional smooth irreducible $E$-linear representation of $\B_2(\Qp)$ having a central character $\chi$, then there exists an irreducible $E$-linear representation $V$ of $\Gal(\Qpbar/\Qp)$ such that $\Pi=\Omega_\chi(V)$. 
\end{enonce*}

Our proof extends to representations with coefficients in an algebraically closed field $E$ of characteristic $p$. The theory of $(\phi,\Gamma)$-modules is then less satisfactory, but one can still carry out Colmez's construction and prove an analogue of theorem A.

\begin{enonce*}{Theorem A'}
If $E$ is an algebraically closed field of characteristic $p$, and if $\Pi$ is an infinite dimensional smooth irreducible $E$-linear representation of $\B_2(\Qp)$ having a central character $\chi$, then there exists an irreducible $E$-linear representation $V$ of the Weil group of $\Qp$ such that $\Pi=\Omega_\chi(V)$. 
\end{enonce*}

This extension of theorem A depends on the following result, which (following a suggestion of Colmez) extends Fontaine's theory of $(\phi,\Gamma)$-modules to algebraically closed coefficient fields of characteristic $p$.

\begin{enonce*}{Theorem B}
If $E$ is an algebraically closed field of characteristic $p$, then there is a natural bijection between the set of irreducible $E$-linear representations of the Weil group of $\Qp$ and the set of irreducible $(\phi,\Gamma)$-modules over $E\dpar{X}$.
\end{enonce*}

This bijection, which is compatible with the usual theory of $(\phi,\Gamma)$-modules, does not seem to extend to reducible objects if $E$ is not an algebraic extension of $\Fp$.

In order to prove theorems A and A', we need to ``invert'' Colmez's construction $V \mapsto (\projlim_\psi D^\natural(V))^*$. This was done in some cases by Colmez (see \S IV of \cite{PCpgm} as well as \S 4 of Emerton's \cite{MEcoh}) and in much greater generality by Schneider and Vign\'eras (see \cite{SVfun}). Our method is similar. The finiteness result that we need in order to conclude is provided by Emerton (see \cite{MEcoh}).

Note that if $E$ is a field of characteristic different from $p$, then determining the smooth irreducible representations of $\B_2(\Qp)$ is a much easier problem (see for instance \S 8 of \cite{BHgl2} for the case $E=\mathbf{C}$). Likewise, it is a simple exercise to determine the finite dimensional smooth irreducible $E$-linear representations of $\B_2(\Qp)$.

\subsection*{Notation}
The letter $E$ stands for a field of characteristic $p$. Throughout this article, $E$ is given the discrete topology. We let $\B=\B_2(\Qp)$ and write $\gal$ for $\Gal(\Qpbar/\Qp)$. We define a map $n : \gal \to \hat{\ZZ}$ as follows: if $g \in \gal$, then the image of $g$ in $\Gal(\Fpbar/\Fp)$ is $\Frob_p^{n(g)}$ where $\Frob_p = [x \mapsto x^p]$.  The Weil group of $\Qp$ is $\weil = \{ g \in \gal$ such that $n(g) \in \ZZ\}$ and $\inert$ denotes the inertia subgroup of $\gal$.

In order to retain the spirit of the lectures given at the LMS Durham Symposium, we explain the idea of the proofs of some of the technical results that are taken from other papers, in order for this article to be more easily readable by newcomers to the subject. 

\section{$(\phi,\Gamma)$-modules and $(\psi,\Gamma)$-modules}\label{phigam}

In this section, we recall the definition of $(\phi,\Gamma)$-modules and $(\psi,\Gamma)$-modules and we explain how these objects are related to each other. 

The ring $E\dcroc{X}$ is given the $X$-adic topology, for which it is complete, and the field $E\dpar{X} = \cup_{n \geq 0} X^{-n} E\dcroc{X}$ is given the inductive limit topology when necessary.

The rings $E\dcroc{X}$ and $E\dpar{X}$ are equipped with a continuous Frobenius map $\phi$ given by $(\phi f)(X)=f(X^p)$. Let $\Gamma$ stand for the group $\Zp^\times$, the element of $\Gamma$ corresponding to $a \in \Zp^\times$ being denoted by $[a]$. The rings $E\dcroc{X}$ and $E\dpar{X}$ are also equipped with an action of $\Gamma$, given by $([a]f)(X)=f((1+X)^a-1)$. This action is continuous and commutes with $\phi$.

\begin{defi}
\label{defpgm}
A $(\phi,\Gamma)$-module is an $E\dpar{X}$-vector space $\dfont$ of dimension $d$, equipped with a semilinear Frobenius map $\phi : \dfont \to \dfont$ whose matrix in some basis belongs to $\GL_d(E\dpar{X})$, and a continuous semilinear action of $\Gamma$ that commutes with $\phi$.
\end{defi}

\begin{exem}
\label{pgexem}
If $\delta : \Qp^\times \to E^\times$ is a continuous character, then we define $E\dpar{X}(\delta)$ as the $(\phi,\Gamma)$-module of dimension $1$ having $e_\delta$ as a basis, where $\phi(e_\delta) = \delta(p) e_\delta$ and $[a]e_\delta = \delta(a)e_\delta$. Every $(\phi,\Gamma)$-module of dimension $1$ is then isomorphic to $E\dpar{X}(\delta)$ for a well-defined character $\delta : \Qp^\times \to E^\times$.
\end{exem}

If $\alpha(X) \in E\dpar{X}$, then we can write $\alpha(X) = \sum_{j=0}^{p-1} (1+X)^j \alpha_j(X^p) $ in a unique way, and we define a map $\psi : E\dpar{X} \to E\dpar{X}$ by the formula $\psi(\alpha)(X) = \alpha_0(X)$. A direct computation shows that if $0 \leq r \leq p-1$ then $\psi(X^{pm+r}) = (-1)^r X^m$.

If $\dfont$ is a $(\phi,\Gamma)$-module over $E\dpar{X}$ and if $y \in \dfont$, then we can write as above $y= \sum_{j=0}^{p-1} (1+X)^j \phi(y_j)$, and we set $\psi(y) = y_0$. The operator $\psi$ thus defined commutes with the action of $\Gamma$ and satisfies $\psi(\alpha(X) \phi(y)) = \psi(\alpha)(X) y$ and $\psi(\alpha(X^p)y) = \alpha(X) \psi(y)$.

\begin{defi}
\label{defpsigam}
A $(\psi,\Gamma)$-module is an $E\dcroc{X}$-module $\mfont$ of finite type, equipped with a map $\psi : \mfont \to \mfont$ such that $\psi( f(X^p) y ) = f(X) \psi(y)$, and a continuous semilinear action of $\Gamma$ that commutes with $\psi$. We say that 
\begin{enumerate}
\item $\mfont$ is surjective if $\psi : \mfont \to \mfont$ is surjective;
\item $\mfont$ is non-degenerate if $\ker(\psi : \mfont \to \mfont)$ does not contain an $E\dcroc{X}$-submodule (in other words: if $y \in \mfont$ satisfies $\psi(f(X)y)=0$ for all $f(X) \in E\dcroc{X}$, then $y=0$);
\item $\mfont$ is irreducible if it has no non-trivial sub-$(\psi,\Gamma)$-module.
\end{enumerate}
\end{defi}

Note that an irreducible $(\psi,\Gamma)$-module is surjective and non-degenerate. It is also torsion-free unless it is finite-dimensional over $E$.

\begin{theo}
\label{phtops}
If $\dfont$ is a $(\phi,\Gamma)$-module, then $\dfont$ contains a surjective sub-$(\psi,\Gamma)$-module $\mfont$ such that $\dfont = E\dpar{X} \otimes_{E\dcroc{X}} \mfont$. In addition,
\begin{enumerate}
\item if $\dfont$ is irreducible of dimension $\geq 2$, then $\mfont$ is uniquely determined;
\item if $\dfont$ is of dimension $1$, and we write $\dfont=E\dpar{X}(\delta)$, then either $\mfont=E\dcroc{X} \cdot e_\delta$ or $\mfont=X^{-1}E\dcroc{X} \cdot e_\delta$.
\end{enumerate}
\end{theo}

\begin{proof}
This is proved in \S II.4 and \S II.5 of \cite{PCmir}. Note that if $\dfont$ is of dimension $1$, then the existence of $\mfont$ and the fact that either $\mfont=E\dcroc{X}\cdot e_\delta$ or $\mfont=X^{-1}E\dcroc{X}\cdot e_\delta$ are both simple exercises. In general, Colmez constructs both a smallest and a largest such sub-$(\psi,\Gamma)$-module, denoted by $\dfont^\natural$ and $\dfont^\sharp$ respectively. He then proves (corollary II.5.21 of \cite{PCmir}) that if $\dfont$ is irreducible of dimension $\geq 2$, then $\dfont^\natural = \dfont^\sharp$.
\end{proof}

\begin{defi}
\label{mofd}
We denote by $\mfont(\dfont)$ the surjective $(\psi,\Gamma)$-module attached to an irreducible $(\phi,\Gamma)$-module $\dfont$ (if $\dfont$ is of dimension $1$, then we take $\mfont(\dfont)=E\dcroc{X}\cdot e_\delta$), so that our $\mfont(\dfont)$ is Colmez's $\dfont^\natural$.
\end{defi}

\begin{theo}\label{psitophi}
If $M$ is a surjective $(\psi,\Gamma)$-module that is non-degenerate and free over $E\dcroc{X}$, then there exists a compatible $(\phi,\Gamma)$-module structure on $\dfont = E\dpar{X} \otimes_{E\dcroc{X}} \mfont$.
\end{theo}

\begin{proof}
Let $\dfont = E\dpar{X} \otimes_{E\dcroc{X}} \mfont$ and let $\dtilde$ be $\dfont$ but with the $E\dpar{X}$-vector space structure given by $f(X) \cdot y = f(X^p)y$ so that $\dtilde$ is an $E\dpar{X}$-vector space of dimension $pd$. Let $\psi_j : \dfont \to \dfont$ be the map $y \mapsto \psi((1+X)^{-j} y)$, so that $\psi_j : \dtilde \to \dfont$ is a surjective linear map. Its kernel is therefore of dimension $pd-d$ and $\nfont = \cap_{i=1}^{p-1} \ker \psi_j$ is an $E\dpar{X}$-vector space of dimension at least $pd-(p-1)d = d$. The non-degeneracy of $\mfont$ implies that $\psi : \nfont \to \dfont$ is injective, so that $\dim \nfont=d$ and $\psi : \nfont \to \dfont$ is bijective. 

Let $\phi : \dfont \to \nfont \subset \dfont$ denote its inverse. It is easily checked that $\phi$ and $\Gamma$ give rise to a $(\phi,\Gamma)$-module structure on $\dfont$, compatible with the $(\psi,\Gamma)$-module structure on $\mfont$.
\end{proof}

We finish this section with a technical result on regularization by Frobenius. Let $R$ be a ring equipped with an automorphism $\phi$, which is extended to $R\dcroc{X}$ by $\phi(X)=X^p$.

\begin{lemm}\label{phireg}
If $P \in \GL_d(R\dcroc{X})$, then there exists a matrix $M \in \GL_d(R\dcroc{X})$ such that $M^{-1} P \phi(M) = P(0) \in \GL_d(R)$.
\end{lemm}

\begin{proof}
This is a standard result, which is proved by successive approximation: if there exists a matrix $M_i \in \GL_d(R\dcroc{X})$ such that $M_i^{-1} P \phi(M_i) = P(0) + P_i X^i + \mathrm{O}(X^{i+1})$ with $P_i \in \M_d(R)$ and if $Q_i = P_i P(0)^{-1}$, then 
\[ (1+X^i Q_i)^{-1} M_i^{-1} \cdot P \cdot \phi(M_i (1+X^i Q_i)) = P(0) + \mathrm{O}(X^{i+1}), \] 
so that one can set $M_{i+1} = M_i \cdot (1+X^i Q_i)$ and take $M= \lim_{i \to +\infty} M_i$.
\end{proof}

\section{Construction of Galois representations}\label{galrep}

In this section, we recall Fontaine's equivalence between $(\phi,\Gamma)$-modules and  representations of $\gal$ over finite fields. After that, we explain how to extend this equivalence to irreducible representations of $\weil$ over algebraically closed fields.

Let $\efont_{\Qp} = \Fp \dpar{X}$ and recall that if $K$ is a finite Galois extension of $\Qp$, then there exists a finite extension $\efont_K$ of $\efont_{\Qp}$ attached to it by the theory of the field of norms (see \cite{JPWcdn} and A3 of \cite{JMFpg}), and that $\gal$ acts on $\efont_K$. For example, $\gal$ acts on $\efont_{\Qp}$ by $g(f(X)) = f([\chi_{\mathrm{cycl}}(g)](X))$. We have $\efont_{\Qp}^{\mathrm{sep}} = \cup_{K / \Qp} \efont_K$ and if $\hal$ denotes the kernel of $\chi_{\mathrm{cycl}} : \gal \to \Zp^\times$, then the map $\hal \to \Gal(\efont_{\Qp}^{\mathrm{sep}}/\efont_{\Qp})$ is an isomorphism.

If $E$ is a finite field and if $\dfont$ is a $(\phi,\Gamma)$-module over $E\dpar{X}$, then $V(\dfont)=(\efont_{\Qp}^{\mathrm{sep}} \otimes_{\Fp\dpar{X}} \dfont)^{\phi=1}$ is an $E$-vector space and the group $\gal$ acts on $V(\dfont)$ by the formula $g(\alpha \otimes y) = g(\alpha) \otimes [\chi_{\mathrm{cycl}}(g)](y)$. This way, we get a functor from the category of $(\phi,\Gamma)$-modules over $E\dpar{X}$ to the category of $E$-linear representations of $\gal$. The following theorem is proved in \S 1.2 of \cite{JMFpg}.

\begin{theo}\label{fontequiv}
If $\dfont$ is a $(\phi,\Gamma)$-module over $E\dpar{X}$, then $V(\dfont)$ is an $E$-vector space of dimension $\dim(\dfont)$, and the functor $\dfont \mapsto V(\dfont)$ gives rise to an equivalence of categories between the category of $(\phi,\Gamma)$-modules over $E\dpar{X}$ and the category of $E$-linear representations of $\gal$.
\end{theo}

\begin{proof}
We give a sketch of the proof. Assume first that $E=\Fp$ and let $\dfont$ be a $(\phi,\Gamma)$-module over $\Fp\dpar{X}$. If we choose a basis of $\dfont$ and if $\Mat(\phi)=(p_{ij})_{1 \leq i,j \leq \dim(\dfont)}$ in that basis, then the algebra $A = \Fp\dpar{X}[X_1,\hdots,X_{\dim(\dfont)}] / (X_j^p - \sum_i p_{ij} X_i)_{1 \leq j \leq \dim(\dfont)}$ is an \'etale $\Fp\dpar{X}$-algebra of rank $p^{\dim(\dfont)}$ and $V(\dfont)=\operatorname{Hom}_{\Fp\dpar{X}-\text{algebra}}(A,\Fp\dpar{X}^{\text{sep}})$ so that $V(\dfont)$ is an $\Fp$-vector space of dimension $\dim(\dfont)$. 

Given the isomorphism $\hal \simeq \operatorname{Gal}(\Fp\dpar{X}^{\text{sep}} / \Fp\dpar{X})$, Hilbert's theorem 90 tells us that $\mathrm{H}^1_{\text{discrete}}(\hal, \GL_d(\Fp\dpar{X}^{\text{sep}})) = \{ 1 \}$ if $d \geq 1$. If $V$ is an $\Fp$-linear representation of $\hal$ then $\Fp\dpar{X}^{\text{sep}} \otimes_{\Fp} V \simeq  (\Fp\dpar{X}^{\text{sep}})^{\dim(V)}$ as representations of $\hal$ so that the $\Fp\dpar{X}$-vector space $\dfont(V) = (\Fp\dpar{X}^{\text{sep}} \otimes_{\Fp} V)^{\hal}$ is of dimension $\dim(V)$ and $V= (\Fp\dpar{X}^{\text{sep}} \otimes_{\Fp\dpar{X}} \dfont(V))^{\phi=1}$. 

It is then easy to check that the functors $V \mapsto \dfont(V)$ and $\dfont \mapsto V(\dfont)$ are inverse of each other. Finally, if $E\neq \Fp$ then one can consider an $E$-linear representation as an $\Fp$-linear representation with an $E$-linear structure and likewise for $(\phi,\Gamma)$-modules, so that the equivalence carries over.
\end{proof}

For example, if $\delta$ is a character of $\Qp^\times$, then the representation arising from $E\dpar{X}(\delta)$ is the character of $\gal$ corresponding to $\delta$ by local class field theory.

If $E$ is not a finite extension of $\Fp$, then theorem \ref{fontequiv} above may well fail. Suppose for instance that $E=\Fp(t)$ and that $\dfont = E\dpar{X} (\delta)$ where $\delta(p)=t$ and $\delta |_{\Zp^\times} = 1$. This $(\phi,\Gamma)$-module ``should'' correspond to the unramified character of $\gal$ sending $\Frob_p$ to $t^{-1}$, but there is no such character because the map $n \mapsto t^{-n}$ does not extend to $\hat{\ZZ}$. There is however such a character of the Weil group $\weil$ of $\Qp$ and in the rest of this section, we construct a bijection between the set of irreducible representations of $\weil$ and the set of irreducible $(\phi,\Gamma)$-modules over $E\dpar{X}$, for any algebraically closed field $E$.

Assume for the rest of this section that $E$ is an algebraically closed field of characteristic $p$. We first explain how to attach an irreducible $(\phi,\Gamma)$-module over $E\dpar{X}$ to an irreducible $E$-linear representation of $\weil$. If $\lambda \in E^\times$, let $\mu_\lambda : \weil \to E^\times$ denote the character defined by $g \mapsto \lambda^{-n(g)}$. Take $n \geq 1$ and let $h \in \ZZ$ be an integer. Let $W = \{ \alpha \in \Fpbar$ such that $\alpha^{p^n} = (-1)^{n-1} \alpha\}$ so that $W$ is a $\Fpn$-vector space of dimension $1$ and hence a $\Fp$-vector space of dimension $n$. Choose $\pi_n \in \Qpbar$ such that $\pi_n^{p^n-1}= -p$. By composing the map $\Gal(\Qp^\text{nr}(\pi_n) / \Qp) \xrightarrow{\sim} \Fpn^\times \rtimes \hat{\ZZ}$ with the map $\Fpn^\times \rtimes \hat{\ZZ} \to \operatorname{End}_{\Fp}(W)$ given by $(x,0) \mapsto m_x^h$ (where $m_x$ is the multiplication by $x$ map) and by $(1,1) \mapsto (\alpha \mapsto \alpha^p)$, we make $W$ into an $n$-dimensional $\Fp$-linear representation of $\gal$ which we call $\ind(\omega_n^h)$. Its determinant is $\omega^h$ where $\omega : \gal \to \Fp^\times$ is the mod $p$ cyclotomic character. 

If $h$ is not divisible by any of the $(p^n-1)/(p^d-1)$ for $d<n$ dividing $n$ (we then say that $h$ is primitive), then $\ind(\omega_n^h)$ is irreducible and it is uniquely determined by the two conditions $\det \ind(\omega_n^h) = \omega^h$ and $\ind(\omega_n^h) |_{\inert} =  \oplus_{i=0}^{n-1} \omega_n^{p^i h}$ where $\omega_n : \inert \to \Fpbar^\times$  is one of Serre's fundamental characters of level $n$, given by $\omega_n(g) = g(\pi_n)/\pi_n \bmod{p}$ (see for example \cite{JPSof} and \S 2.1 of \cite{LBmod}).

\begin{prop}\label{classrepw}
If $V$ is an irreducible $n$-dimensional $E$-linear representation of $\weil$, then there exists $h \in \ZZ$ and $\lambda \in E^\times$ such that $V = (E \otimes_{\Fp} \ind (\omega_n^h)) \otimes \mu_\lambda$.
\end{prop}

\begin{proof}
The proof is the same as in \S 2.1 of \cite{LBmod}: by \S 1.6 of \cite{JPSof}, $V|_{\inert}$ splits as a direct sum of $n$ tame characters and since $V$ is irreducible, these characters are  transitively permuted by Frobenius, so that they are of level $n$. Therefore, there exists a primitive $h$ such that $V = \oplus_{i=0}^{n-1} V_i$ where $\inert$ acts on $V_i$ by $\omega_n^{p^i h}$. Since $\omega_n$ extends to $\Gal(\Qpbar/\QQ_{p^n})$, each $V_i$ is stable under the Weil group of $\Gal(\Qpbar/\QQ_{p^n})$, which then acts on $V_i$ by $\omega_n^{p^i h} \chi_i$ where $\chi_i$ is an unramified character. The lemma then follows from Frobenius reciprocity. 
\end{proof}

\begin{defi}\label{defwp}
To $V = (E \otimes_{\Fp} \ind (\omega_n^h)) \otimes \mu_\lambda$, we then attach the $(\phi,\Gamma)$-module $\dfont(V)$ having a basis $e_0,\hdots,e_{n-1}$ in which $[a](e_j) =  (aX/[a](X)) ^{hp^j(p-1)/(p^n-1)} e_j$ if $a \in \Zp^\times$ and $\phi(e_j) = e_{j+1}$ for $0 \leq j \leq n-2$ and $\phi(e_{n-1}) = (-1)^{n-1} \lambda^n X^{-h(p-1)} e_0$.
\end{defi}

Different choices of $h$ and $\lambda$ can give rise to the same representation $V$, but we can check that the $(\phi,\Gamma)$-module  $\dfont(V)$ thus defined depends only on $V$. Indeed, if  $\lambda \in \Fpbar$, then $(E \otimes_{\Fp} \ind (\omega_n^h)) \otimes \mu_\lambda$ extends to $\gal$ and the $(\phi,\Gamma)$-module above is the extension of scalars of the one given by Fontaine's construction, by the results of \S 2.1 of \cite{LBmod}.

We now explain how to attach an irreducible representation of $\weil$ to an irreducible $(\phi,\Gamma)$-module over $E\dpar{X}$. Let $F$ be a field that is complete for a discrete valuation $\val$ and endowed with an automorphism $\phi$, such that $\val(\phi(y)) = p \cdot \val(y)$ (in the sequel, we'll have $F=E\dpar{Y}$ where $Y^n=X$ and $\val=\vx$). Let $F\{\phi\}$ denote the non-commutative ring of polynomials in $\phi$ with coefficients in $F$. If $P(\phi) = a_0 + a_1 \phi + \cdots + a_n \phi^n \in F\{\phi\}$, then the Newton polygon $\NP(P)$ of $P$ is the polygon whose support consists of the points $([k],\val(a_k))$ where $[k] = (p^k-1)/(p-1)$. The slopes of $\NP(P)$ are the opposites of the slopes of the segments of the polygon. If $P(\phi) = a_0 + a_1 \phi + \cdots + a_n \phi^n \in F\{\phi\}$, and if $y \in F$, then $P(\phi)y =  a_0 y + a_1 \phi(y) \phi  + \cdots + a_n \phi^n(y) \phi^n$.

\begin{lemm}
\label{chgsl}
If $P(\phi) \in F\{\phi\}$ is isoclinic of slope $s$, and if $y \in F$ satisfies $\val(y)=r$, then $P(\phi)y$ is isoclinic of slope $s-(p-1)r$.
\end{lemm}

\begin{proof}
We have $\val(\phi^k(y) a_k) = p^k \val(y) + \val(a_k)$ so that
\[ \frac{\val(\phi^k(y) a_k) - \val(\phi^j(y) a_j)}{[k]-[j]} = \frac{\val(a_k) - \val(a_j)}{[k]-[j]} + (p-1)r. \]
\end{proof}

\begin{prop}
\label{splitsl}
If $P(\phi) \in F\{\phi\}$ is irreducible, then it is isoclinic.
\end{prop}

\begin{proof}
See \S 2.4 of \cite{KKrel} as well as \cite{MVphd}. We give here a sketch of the proof. Let $F\{\phi^{\pm 1}\}$ be the space of polynomials in $\phi$ and $\phi^{-1}$. Since $\phi : F \to F$ is not necessarily invertible, $F\{\phi^{\pm 1}\}$ is not a ring, but it is a left $F\{\phi\}$-module. If $r \in \RR$ and $P \in F\{\phi^{\pm 1}\}$, let $\vr(P)= \min_{i \in \ZZ} (\val(a_i) + r[i])$. Using successive approximations, we can show that if $R \in F\{\phi^{\pm 1}\}$ and $r \in \RR$ are such that $\vr(R-1)>0$, then there exists $P \in F\{\phi\}$ and $Q \in F\{\phi^{-1}\}$ such that $R=PQ$. Using this factorization result, we can now prove that if $P \in F\{\phi\}$ and $\NP(P)$ has a breakpoint, then $P$ can be factored in $F\{\phi\}$.
\end{proof}

Note that in general, if $P=P_1P_2$, then the set of slopes of $\NP(P)$ is not the union of the sets of slopes of $\NP(P_1)$ and $\NP(P_2)$. 

We denote by $\vx$ the $X$-adic valuation on $\efont_K$, by $\efont_K^+$ the ring of integers of $\efont_K$ for $\vx$ and by $k_K$ the residue field of $\efont_K$ (it is the residue field of $K(\mu_{p^\infty})$).

\begin{prop}
\label{bascal}
If $\dfont$ is an irreducible $(\phi,\Gamma)$-module over $E\dpar{X}$, then there exists a finite extension $K$ of $\Qp$, such that $\efont_K \otimes_{\efont_{\Qp}} \dfont$ has a basis in which $\Mat(\phi)$ belongs to $\GL_d(k_K \otimes_{\Fp} E)$.

The $k_K \otimes_{\Fp} E$-module generated by this basis depends only on $\dfont$, and in particular it is stable under the action of $\gal$ given by $g(\alpha \otimes y) = g(\alpha) \otimes [\chi_{\mathrm{cycl}}(g)](y)$. 
\end{prop}

\begin{proof}
Let us first show that the $k_K \otimes_{\Fp} E$-module generated by such a basis is unique. If $M \in \M_d(\efont_K \otimes_{\efont_{\Qp}} E \dpar{X})$, then let $\vx(M)$ be the minimum of the valuations of the entries of $M$.

If $\efont_K \otimes_{\efont_{\Qp}} \dfont$ admits two bases in which $\Mat(\phi) \in \GL_d(k_K \otimes_{\Fp} E)$, then let $P_1$ and $P_2$ be the two matrices of $\phi$ and let $B \in \GL_d(\efont_K \otimes_{\efont_{\Qp}} E \dpar{X})$ be the change of basis matrix. We then have $P_2 = B^{-1} P_1 \phi(B)$ so that $\phi(B) = P_1^{-1} B P_2$. This implies that $\vx(\phi(B)) = \vx(B)$ so that $\vx(B) = 0$, and hence $B \in \M_d(\efont^+_K \otimes_{\efont^+_{\Qp}} E \dcroc{X})$. The same argument applied to $B^{-1}$ shows that $B \in \GL_d(\efont^+_K \otimes_{\efont^+_{\Qp}} E \dcroc{X})$. If we write $B=B_0 + C$ where $B_0 \in \GL_d(k_K \otimes_{\Fp} E)$ and $\vx(C)>0$, then the formula $\phi(B) = P_1^{-1} B P_2$ implies likewise that $\vx(C) = +\infty$ so that $C=0$. The $k_K \otimes_{\Fp} E$-module generated by these two bases is therefore the same.

We now show the existence of such a basis. We can assume that $\dfont$ is irreducible as a $\phi$-module; indeed, if $\mfont$ is an irreducible sub-$\phi$-module of $\dfont$, then we can write $\dfont=\sum_{i=1}^n \gamma_i(\mfont)$ with $\gamma_i \in \Gamma$. We can assume that $n$ is minimal, so that the sum is direct and the existence result for $\dfont$ follows from the result for each of the $\phi$-modules $\gamma_i(\mfont)$.

If $m \in \dfont$ is non-zero, then it generates $\dfont$ as an $E\dpar{X}\{\phi\}$-module since $\dfont$ is assumed to be irreducible. Let $P(\phi)$ be a non-zero polynomial of degree $\dim \dfont$ such that $P(\phi)(m)=0$. If $P(\phi)$ were reducible, then this would correspond to a non-trivial sub-$\phi$-module of $\dfont$ so that $P(\phi)$ is irreducible and by proposition \ref{splitsl}, $P(\phi)$ is isoclinic. If $s$ is the slope of $P(\phi)$, then there exists a finite extension $K$ of $\Qp$ and an element $y \in \efont_K$ of valuation $s/(p-1)$. Lemma \ref{chgsl} shows that if we replace $m$ by $ym$, then the resulting polynomial $Q(\phi)$ is isoclinic of slope $0$. This implies that there exists a basis of $\efont_K \otimes_{\efont_{\Qp}} \dfont$ in which $\Mat(\phi) \in \GL_d((k_K \otimes_{\Fp} E) \dcroc{Y})$. Lemma \ref{phireg} now implies that there exists a basis of $\efont_K \otimes_{\efont_{\Qp}} \dfont$ in which $\Mat(\phi) \in \GL_d(k_K \otimes_{\Fp} E)$, which is the sought-after result.
\end{proof}

Let $\dfont$ be an irreducible $(\phi,\Gamma)$-module over $E\dpar{X}$, and let $K$ be as above. Since $E$ is algebraically closed, we have $k_K \otimes_{\Fp} E = E^n$ with $n=[k_K:\Fp]$. We denote by $\pi_k : E^n \to E$ the projection on the $k$-th factor. Let $V_K(\dfont)$ be the $E^n$-module generated by the basis afforded by proposition \ref{bascal}. This module is stable under $\gal$ which acts by $k_K$-semilinear automorphisms. We define an action of $\weil$ on $V_K(\dfont)$ by $\rho(g)(y) = \phi^{-n(g)}(g(y))$. This action is now $E^n$-linear, and commutes with $\phi$. In particular, $V_K(\dfont) = \pi_1 V_K(\dfont) \oplus \cdots \oplus \pi_n V_K(\dfont)$ and $\phi(\pi_k V_K(\dfont)) = \pi_{k+1} V_K(\dfont)$ (with $\pi_{n+1}=\pi_1$) so that all the representations $\pi_k V_K(\dfont)$ are isomorphic. We let $V(\dfont) = \pi_1 V_K(\dfont)$.

\begin{prop}
\label{vdirred}
The representation $V(\dfont)$ defined above is irreducible.
\end{prop}

\begin{proof}
Note that $\phi^n$ gives rise to an endomorphism of $V(\dfont)$. Since $E$ is algebraically closed, $\phi^n$ has an eigenvalue $\lambda$, and the space $V(\dfont)^{\phi^n = \lambda}$ is stable under $\weil$, so that it contains an irreducible sub-representation $W$ of $\weil$.

The $k_K \otimes_{\Fp} E$-module $M = W \oplus \phi(W) \oplus \cdots \oplus \phi^{n-1}(W)$ is then a subspace of $V_K(\dfont)$, which is stable under $\weil$ and $\phi$, so that it is also stable under $\gal$ and $\phi$. The space $\efont_K \otimes_{k_K} M$ is then a sub-$(\phi,\Gamma)$-module of $\efont_K \otimes_{\efont_{\Qp}} \dfont$ that is stable under $\phi$ and $\gal$. By Galois descent (see for instance proposition 2.2.1 of \cite{BCfam}), $\efont_K \otimes_{k_K} M$ comes by extension of scalars from a sub-$(\phi,\Gamma)$-module of $\dfont$. If $\dfont$ is irreducible, then $M=\{0\}$ or $M=V_K(\dfont)$ and hence $V(\dfont)$ is irreducible.
\end{proof}

\begin{theo}\label{equiweil}
The two constructions $V \mapsto \dfont(V)$ and $\dfont \mapsto V(\dfont)$ defined above are inverse of each other and give rise to dimension preserving bijections between the set of irreducible $E$-linear representations of $\weil$ and the set of irreducible $(\phi,\Gamma)$-modules over $E\dpar{X}$.
\end{theo}

\begin{proof}
The fact that dimensions are preserved is clear from the constructions. The fact that the two constructions are inverse of each other is a tedious but straightforward exercise.
\end{proof}

\section{Topological representations of profinite groups}\label{profreps}

In this section, we first gather some results about topological $E$-vector spaces and duality, which generalize Pontryagin's theorems to certain $E$-vector spaces (see \S II.6 of \cite{SLtop}). After that, we look at continous representations of certain topological groups.

Recall that $E$ is a field that is taken with the discrete topology. A topological $E$-vector space $V$ is said to be linearly topologized if $V$ is separated (Hausdorff) and if $\{0\}$ has a basis of neighborhoods that are all vector spaces. For example, the discrete topology on $V$ is a linear topology. We denote by $\disc$ the category whose objects are the $E$-linear vector spaces with the discrete topology, with continuous linear maps as morphisms.

We say that an affine subspace $W$ of a linearly topologized $E$-vector space $V$ is linearly compact if every family $\{W_i\}_{i \in I}$ of closed affine subspaces of $W$ having the finite intersection property has a non-empty intersection. Linearly compact affine spaces generally enjoy the same properties as compact topological spaces (see (27) of \S II.6 of \cite{SLtop}). For example, a linearly compact subspace of $V$ is closed in $V$, its image under a continuous linear map is linearly compact, and if $V_1,\hdots,V_n$ are linearly compact spaces, then their product is linearly compact. If $V$ is linearly compact and if $W$ is a closed subspace of $V$, then $W$ is open in $V$ if and only if it is of finite codimension.

We say that an $E$-vector space is of profinite dimension if it is an inverse limit of finite dimensional discrete $E$-vector spaces. For example, $E\dcroc{X}$ with the $X$-adic topology is of profinite dimension. Such a space is then linearly compact and conversely, by (32) of \S II.6 of \cite{SLtop}, linearly compact spaces are profinite dimensional. We denote by $\comp$ the category whose objects are the linearly compact $E$-vector spaces, with continuous linear maps as morphisms.

If $V$ is a topological vector space, we denote by $V^\ast$ its continuous dual. This space is given a linear topology by choosing as a basis of neighborhoods of $\{0\}$ the set $\{E^\perp\}_E$ where $E$ runs through all linearly compact subspaces of $V$, and $E^\perp = \{ f \in V^*$ such that $f(v)=0$ for all $v \in E\}$.

\begin{theo}
\label{dualprodis}
The duality functor $V \mapsto V^*$ gives rise to equivalences of categories $\disc \to \comp$ and $\comp \to \disc$. 

Moreover, the natural map $V \to (V^*)^*$ is an isomorphism.
\end{theo}

\begin{proof}
See (29) in \S II.6 of \cite{SLtop}.
\end{proof}

We now turn to group representations. Let $G$ be a topological group and let $\discg$ and $\compg$ denote the categories of continuous $E$-linear representations of $G$ on either discrete or linearly compact spaces. If $V$ is a representation of $G$, then $V^*$ is a representation of $G$, with the usual action given by $(gf)(v) = f(g^{-1}v)$.

\begin{prop}\label{dualirred}
If $V \in \discg$ or $V \in \compg$ is topologically irreducible, then so is its dual $V^*$.
\end{prop}

\begin{proof}
If $W$ is a closed subspace of $V^*$ stable under $G$, then let $W^\perp = \{ v \in V$ such that $f(v)=0$ for all $f \in W\}$. The natural map $W^\perp \to (V^*/W)^*$ is an isomorphism by theorem \ref{dualprodis}. Moreover, $W^\perp$ is a closed subspace of $V$, that is also stable under $G$, so that either $W^\perp=\{0\}$ and $W=V^*$ or $W^\perp=V$ and $W=\{0\}$.
\end{proof}

Assume now that $G$ is a topologically finitely generated profinite group (in this article we only need the case $G=\Zp$). Denote by $V(G)$ the sub-$E$-vector space of $V$ generated by the elements $(g-1)v$ where $g \in G$ and $v \in V$.

\begin{prop}\label{cljm}
If $V \in \compg$, then $V(G)$ is a closed subspace of $V$.
\end{prop}

\begin{proof}
Let $g_1,\hdots,g_n$ be elements generating a dense subgroup $G'$ of $G$. The subspace $(g_i-1)V$ is the image of a linearly compact subspace by a continuous linear map and is hence linearly compact. This implies that $V(G') = \sum_{i=1}^n (g_i-1)V$ is linearly compact and therefore closed in $V$. 

If $v \in V$, then the image of $G'$ under the map $g \mapsto (g-1)v$ is contained in $V(G')$ and, since $G'$ is dense in $G$ and $V(G')$ is closed in $V$, the image of $G$ is also contained in $V(G')$ so that $V(G)=V(G')$ and $V(G)$ is closed in $V$.
\end{proof}

Note that the same is trivially true if $V \in \discg$. We set $V_G = V/V(G)$.

\begin{prop}\label{coinv}
If $V \in \discg$ or $V \in \compg$, then $(V^G)^* = (V^*)_G$.
\end{prop}

\begin{proof}
If $f\in V^*$, then $f(gv)=f(v)$ for all $g \in G$ and $v \in V$ if and only if $f$ is zero on $V(G)$. This implies that $(V^*)^G = (V_G)^*$. Replacing $V$ by $V^*$ in this formula and dualizing gives us the proposition.
\end{proof}

Let $E\dcroc{G} = \projlim_N E[G/N]$ denote the completed group algebra of $G$, where $N$ runs through the set of open normal subgroups of $G$.

\begin{prop}
\label{egmod}
If $V \in \compg$ or $V \in \discg$, then $V$ is an $E\dcroc{G}$-module.
\end{prop}

\begin{proof}
If $V \in \discg$, then this is immediate, so assume that $V \in \compg$. The space $V$ is a projective limit of finite dimensional $E$-vector spaces. We first show that if $V \in \compg$, then $V$ is a projective limit of finite dimensional $E$-linear representations of $G$. It is enough to prove that if $W$ is an open subspace of $V$, then it contains an open subspace stable under $G$. By continuity, for each $g \in G$, there exists an open neighborhood $H_g$ of $g$ in $G$ and an open subspace $W_g$ of $V$ such that $H_g \cdot W_g \subset W$. By compacity of $G$, there exists $g_1,\hdots,g_n \in G$ such that $G = H_{g_1} \cup \cdots \cup H_{g_n}$ and if we set $X = W_{g_1} \cap \cdots \cap W_{g_n}$, then $X$ is an open subspace of $W$ and $G \cdot X \subset W$. The vector space generated by $G \cdot X$ is then open in $W$ and stable under $G$.

Since $V$ is a projective limit of finite dimensional $E$-linear representations of $G$ by the above, and since each of them is an $E\dcroc{G}$-module, then so is $V$.
\end{proof}

We now assume that $G=\Zp$ so that $E\dcroc{G}=E\dcroc{X}$. The following result is a variant of Nakayama's lemma.

\begin{theo}\label{naklem}
If $V \in \comp$ is a topological $E\dcroc{X}$-module, then $V$ is finitely generated over $E\dcroc{X}$ if and only if $V/XV$ is a finite dimensional $E$-vector space.
\end{theo}

\begin{proof}
The fact that if $V$ is finitely generated over $E\dcroc{X}$, then $V/XV$ is a finite dimensional $E$-vector space is immediate, so let us prove the converse.

Let $v_1,\hdots,v_n$ be elements of $V$ that generate $V/XV$ over $E$, and let $W$ be the $E\dcroc{X}$-module generated by $v_1,\hdots,v_n$. The $E$-vector space $W$ is linearly compact, and therefore so is $V/W$. In addition, $(V/W)/X=\{0\}$. It is therefore enough to show that if $V \in \comp$ is a topological $E\dcroc{X}$-module such that $V/XV=\{0\}$, then $V=\{0\}$. 

Let $U$ be an open subspace of $V$. By continuity, there exists an open subspace $W$ of $U$ and $k_0 \geq 1$ such that $X^k W \subset U$ if $k \geq k_0$. Since $W$ is open, it is of finite codimension in $V$ and there exists $v_1,\hdots,v_n \in V$ such that $V = W+Ev_1 + \cdots + Ev_n$. For each $i$, there exists $k_i$ such that $X^k v_i \in U$ if $k \geq k_i$. If $k \geq \max(k_0,\hdots,k_n)$, then $X^k V \subset U$. But $X^k V = V$ so that the only open subspace of $V$ is $V$ itself and hence $V=\{0\}$.
\end{proof}

\section{Colmez's functor}\label{colfun}

In this section, we recall Colmez's construction of representations of $\B$ starting from Galois representations (see \S III of \cite{PCmir}).

If $\mfont$ is a $(\psi,\Gamma)$-module, then we denote by $\projlim_\psi \mfont$ the set of sequences $\{m_n\}_{n \in \ZZ}$ where $m_n \in \mfont$ and $\psi(m_{n+1}) = m_n$ for all $n \in \ZZ$. Let $\chi : \Qp^\times \to E^\times$ be a smooth character. We endow $\projlim_\psi \mfont$ with an action of $\B$ in the following way
\begin{align*}
\left( \smat{ z &  0 \\ 0 & z} \cdot y \right)_i & = \chi(z)^{-1} y_i; \\
\left( \smat{ 1 &  0 \\ 0 & p} \cdot y \right)_i & = y_{i-1} = \psi(y_i); \\
\left( \smat{ 1 &  0 \\ 0 & a} \cdot y \right)_i & = [a^{-1}](y_i); \\
\left( \smat{ 1 &  z \\ 0 & 1} \cdot y \right)_i & = \psi^j((1+X)^{p^{i+j} z} y_{i+j}),\text{ for $i+j \geq -{\rm val}(z)$.}
\end{align*}
It is straightforward to check that these formulas give rise to an action of $\B$, and make $\projlim_\psi \mfont$ into a profinite dimensional topological representation, $\mfont$ itself being separated and complete for the $X$-adic topology (warning: the normalization for the central character is the one chosen in \S 1.2 of \cite{LBgal} and it differs from the one in \S 2.2 of \cite{LBmod}). Note that if $\mfont_1$ and $\mfont_2$ are two $(\psi,\Gamma)$-modules, and there is a map $\mfont_1 \to \mfont_2$, then there is a map $\projlim_\psi \mfont_1 \to \projlim_\psi \mfont_2$. 

\begin{prop}\label{colsubmir}
If $\Sigma$ is a closed subspace of $\projlim_\psi \mfont$ stable under $\B$, then there exists a surjective sub-$(\psi,\Gamma)$-module $\nfont$ of $\mfont$ such that $\Sigma = \projlim_\psi \nfont$.
\end{prop}

\begin{proof}
This is lemma III.3.6 of \cite{PCmir}. We recall the idea of the proof: if $\nfont_k$ is the set of $m \in \mfont$ such that there exists $x \in \Sigma$ with $m = x_k$, then Colmez shows that $\nfont_k$ is a $(\psi,\Gamma)$-module that is independent of $k$ and that we can take $\nfont = \nfont_k$.
\end{proof}

\begin{theo}\label{subpsirred}
If $\Sigma$ is an infinite dimensional topologically irreducible subrepresentation of $\projlim_\psi \mfont$ for some $(\psi,\Gamma)$-module $\mfont$, then there exists a $(\psi,\Gamma)$-module $\nfont$ that is irreducible and free over $E\dcroc{X}$, such that $\Sigma = \projlim_\psi \nfont$.
\end{theo}

\begin{proof}
Let $\mtor$ denote the torsion submodule of $\mfont$. We then have an exact sequence $\projlim_\psi \mtor \to \projlim_\psi \mfont \to \projlim_\psi \mfont / \mtor$. If the image of $\Sigma$ in $\projlim_\psi \mfont/\mtor$ is non-zero, then we have reduced to the case where $\mfont$ is torsion-free. 

Otherwise, $\Sigma$ injects in $\projlim_\psi \mtor$ and $\mtor$ is a finite dimensional $E$-vector space. Proposition \ref{colsubmir} shows that $\Sigma = \projlim_\psi \nfont$ where $\nfont$ is a finite dimensional $E$-vector space. Since $\psi : \nfont \to \nfont$ is surjective, it is injective, and then $\projlim_\psi \nfont = \nfont$ so that $\Sigma$ itself is a finite dimensional $E$-vector space.

We can therefore assume that $\mfont$ is torsion free. Let $\mfont$ be such that $\Sigma$ injects in $\projlim_\psi \mfont$, with $\mfont$ torsion free, surjective and of minimal rank. If $\nfont$ is a sub-$(\psi,\Gamma)$-module of $\mfont$, then the same argument as above shows that $\Sigma$ injects in either $\projlim_\psi \nfont$ or $\projlim_\psi \mfont/\nfont$. This implies that the rank of $\nfont$ is equal to the rank of $\mfont$, so there exists $n \geq 0$ such that $X^n \mfont \subset \nfont$. Repeatedly applying $\psi$ shows that $X \mfont \subset \nfont$. Since $\mfont/X$ is a finite dimensional $E$-vector space, there is therefore a smallest $\mfont$ such that $\Sigma$ injects in $\projlim_\psi \mfont$, and this $\mfont$ is then irreducible.
\end{proof}

If $V$ is an irreducible representation of either $\gal$ (when $E$ is a finite field) or $\weil$ (when $E$ is an algebraically closed field), then by the results of \S \ref{galrep}, we can attach to it a $(\phi,\Gamma)$-module $\dfont(V)$ and then by definition \ref{mofd} an irreducible $(\psi,\Gamma)$-module $\mfont(V)=\mfont(\dfont(V))$. Let $\chi$ be a smooth character of $\Qp^\times$. The space $\projlim_\psi \mfont(V)$ is of profinite dimension and gives rise to a continuous representation of $\B$, which is topologically irreducible by proposition \ref{colsubmir}. Its dual $\Omega_\chi(V) = (\projlim_\psi \mfont(V))^*$ is therefore a smooth irreducible representation of $\B$, with central character $\chi$. We finish by recalling a result of \cite{LBgal} to the effect that $\Omega_\chi(V)$ determines $\chi$ and $V$.

\begin{prop}\label{schurbor}
If $V_1$ and $V_2$ are irreducible and $\Omega_{\chi_1}(V_1)$ is isomorphic to $\Omega_{\chi_2}(V_2)$ as representations of $\B$, then $\chi_1 = \chi_2$ and $V_1=V_2$.
\end{prop}

\begin{proof}
This is proposition 1.2.3 of \cite{LBgal} in the case that $E$ is a finite field, and the proof is similar if $E$ is algebraically closed. We recall the main ideas: since $\chi$ is the central character of  $\Omega_{\chi}(V)$, it is immediate that $\chi_1 = \chi_2$ so we need to show that if there is an equivariant map $f : \projlim_\psi \mfont(V_1) \to \projlim_\psi \mfont(V_2)$, then $V_1 = V_2$. Let $\pr_k : \projlim_\psi \mfont \to \mfont$ denote the map $\{m_n\}_{n \in \ZZ} \mapsto m_k$. If $n \geq 0$, let $K_n$ be the set of elements $m$ of $\projlim_\psi \mfont(V_1)$ such that $\pr_k(m) = 0$ for $k \leq n$. The module $K_n$ is a closed sub-$E\dcroc{X}$-module of $\projlim_\psi \mfont(V_1)$ that is stable under $\psi$ and $\Gamma$, and $\psi(K_n) = K_{n+1}$. This implies that $\pr_0 \circ f(K_n)$ is a sub-$(\psi,\Gamma)$-module of $\mfont(V_2)$. Since $\mfont(V_2)$ is irreducible, we have either $\pr_0 \circ f(K_n)=\{0\}$ or $\pr_0 \circ f(K_n)=\mfont(V_2)$. In addition, $\psi(\pr_0 \circ f(K_n)) = \pr_0 \circ f(K_{n+1})$ and $\pr_0 \circ f(K_n) = \{0\}$ for $n \gg 0$ by continuity, so that $\pr_0 \circ f(K_n) = \{0\}$ for all $n \geq 0$. This implies that $\pr_0 \circ f(m)$ depends only on $m_0$.

The map $m_0 \mapsto \pr_0 \circ f(m)$ from $\mfont(V_1)$ to $\mfont(V_2)$ is therefore a well-defined map of $(\psi,\Gamma)$-modules, which is non-zero because $f$ is an isomorphism. By proposition II.3.4 of \cite{PCmir}, it extends to a map $\dfont(V_1) \to \dfont(V_2)$ so that $\dfont(V_1) \simeq \dfont(V_2)$  and $V_1 \simeq V_2$.
\end{proof}

\section{Representations of $\B_2(\Qp)$}\label{repb}

In this section, we prove that every infinite dimensional smooth irreducible representation of $\B$ having a central character is of the form $\Omega_\chi(V)$ for some $V$ and $\chi$. We start by studying representations of $\B$. Let $\Z = \{ a \cdot \Id$, $a \in \Qp^\times\}$ be the center of $\B$, and let 
\[ \K = \B_2(\Zp) = \pmat{\Zp^\times & \Zp \\ 0 & \Zp^\times}. \]
If $\beta \in \Qp$ and $\delta \in \ZZ$, let
\[ g_{\beta,\delta} = \pmat{ 1 & \beta \\ 0 & p^\delta}. \]
Let $A=\{ \alpha_n p^{-n} + \cdots + \alpha_1 p^{-1}$ where $0 \leq \alpha_j \leq p-1 \}$ so that $A$ is a system of representatives of $\Qp/\Zp$. The following is lemma 1.2.1 of \cite{LBmod}.

\begin{lemm}\label{cosets}
We have $\B = \coprod_{\beta \in A, \delta\in \ZZ} g_{\beta,\delta} \cdot \K\Z$.
\end{lemm}

If $\sigma_1$ and $\sigma_2$ are two smooth characters $\sigma_i : \Qp^\times \to E^\times$, then let $\sigma = \sigma_1 \otimes \sigma_2 : \B \to E^\times$ be the character $\sigma : \smat{a & b \\ 0 & d} \mapsto \sigma_1(a)\sigma_2(d)$ and let $\ind_{\K\Z}^{\B}\sigma$ be the set of functions $f : \B \to E$ satisfying $f(kg) = \sigma(k)f(g)$ if $k \in \K\Z$ and such that $f$ has compact support modulo $\Z$. If $g \in \B$, denote by $[g]$ the function $[g] : \B \to E$ defined by $[g](h)=\sigma(hg)$ if $h \in \K\Z g^{-1}$ and $[g](h)=0$ otherwise. Every element of $\ind_{\K\Z}^{\B}\sigma$ is a finite linear combination of some functions $[g]$. We make $\ind_{\K\Z}^{\B}\sigma$ into a representation of $\B$ in the usual way: if $g \in \B$, then $(gf)(h)=f(hg)$. In particular, we have $g [h] = [gh]$ in addition to the formula $[gk]=\sigma(k)[g]$ for $k\in\K\Z$.

\begin{theo}\label{allquot}
If $\Pi$ is a smooth irreducible representation of $\B$ having a central character, then there exists $\sigma = \sigma_1 \otimes \sigma_2$ such that $\Pi$ is a quotient of $\ind_{\K\Z}^{\B}\sigma$.
\end{theo}

\begin{proof}
This is theorem 1.2.3 of \cite{LBmod}; we recall the proof here. The group $\I$ defined by
\[ \I = \pmat{1+p\Zp & \Zp \\ 0 & 1+p\Zp} \]
is a pro-$p$-group and hence $\Pi^{\I} \neq 0$. Furthermore, $\I$ is a normal subgroup of $\K$ so that $\Pi^{\I}$ is a representation of $\K/\I = \Fp^\times \times \Fp^\times$. Since this group is a finite group of order prime to $p$, we have $\Pi^{\I} = \oplus_{\eta} \Pi^{\K=\eta}$ where $\eta$ runs over the characters of $\Fp^\times \times \Fp^\times$ and since $\Z$ acts through a character by hypothesis, there exists a character $\sigma$ of $\K\Z$ and $v \in \Pi$ such that $k \cdot v = \sigma(k)v$ for $k\in\K\Z$. By Frobenius reciprocity, we get a non-trivial map $\ind_{\K\Z}^{\B}\sigma \to \Pi$ and this map is surjective since $\Pi$ is irreducible.
\end{proof}

Write $\sigma = \sigma_1 \otimes \sigma_2$. We can always assume that we have $\sigma_2(p)=1$, which we now do.

By lemma \ref{cosets}, each $f \in \ind_{\K\Z}^{\B}\sigma$ can be written as $f = \sum_{\beta \in A,\delta \in \ZZ} \alpha(\beta,\delta) [g_{\beta,\delta}]$. 

\begin{defi}\label{defmaps}
Let $s : \ind_{\K\Z}^{\B}\sigma \to E$ be the map \[ s : \sum_{\beta \in A,\delta \in \ZZ} \alpha(\beta,\delta) [g_{\beta,\delta}] \mapsto \sum_{\beta \in A,\delta \in \ZZ} \alpha(\beta,\delta). \]
\end{defi}

The following lemma results from a straightforward calculation (recall that $\sigma_2(p)=1$).

\begin{lemm}\label{indtoe}
The map $s : \ind_{\K\Z}^{\B}\sigma \to E(\sigma)$ is $\B$-equivariant.
\end{lemm}

Let $\B^+$ and $\B^-$ denote the monoids
\[ \B^+ = \left\{ \pmat{p^{\ZZ_{\geq 0}} \Zp^\times & \Zp \\ 0 & \Zp^\times }\right\} \subset \B,\qquad \B^- = \left\{ \pmat{\Zp^\times & \Zp \\ 0 & p^{\ZZ_{\geq 0}} \Zp^\times }\right\} \subset \B,\]
and let $(\ind_{\K\Z}^{\B}\sigma)^+$ denote the set of elements of $\ind_{\K\Z}^{\B}\sigma$ with support in $\B^+$. Since
\[ \pmat{p^n a & b \\ 0 & d} = \pmat{ 1 & p^{-n} bd^{-1} \\ 0 & p^{-n} } \pmat{a&0\\0&d}\pmat{p^n&0\\0&p^n}, \]
$(\ind_{\K\Z}^{\B}\sigma)^+$ is the set of $f=\sum \alpha(\beta,\delta) [g_{\beta,\delta}]$ with $\delta \leq 0$ and $\beta \in p^{-\delta} \Zp / \Zp$.

\begin{lemm}\label{imxkers}
If $y = \sum_{\beta \in A,\delta \in \ZZ} \alpha(\beta,\delta) [g_{\beta,\delta}] \in (\ind_{\K\Z}^{\B}\sigma)^+$, then $y \in X \cdot (\ind_{\K\Z}^{\B}\sigma)^+$ if and only if $\sum_{\beta \in A} \alpha(\beta,\delta)=0$ for all $\delta \leq 0$.
\end{lemm}

\begin{proof}
We have $\smat{1&1\\0&1} [g_{\beta,\delta}] = [g_{\beta+p^{-\delta},\delta}]$ so that 
\[ X \cdot  \sum_{\beta \in A,\delta \in \ZZ} \alpha(\beta,\delta) [g_{\beta,\delta}] =  \sum_{\beta \in A,\delta \in \ZZ} (\alpha(\beta-p^{-\delta},\delta)-\alpha(\beta,\delta)) [g_{\beta,\delta}]. \]  
Since $\beta \in p^{-\delta} \Zp / \Zp$, the lemma follows from the fact that the image of the map $(x_i)_i \mapsto (x_{i-1}-x_i)_i$ from $E^{\ZZ/p^\delta\ZZ}$ to itself is the set of sequences $(x_i)_i$ with $\sum_i x_i = 0$.
\end{proof}

Write $F=\smat{p&0\\0&1}$ and $X=\smat{1&1\\0&1}-\Id$ so that $A=E\dcroc{X}$ is the completed group ring of $\smat{1&\Zp\\0&1}$ and let $A\{F\}$ be the non-commutative ring of polynomials in $F$ with coefficients in $A$, where $FX=X^pF$. If $\Pi = \ind_{\K\Z}^{\B}\sigma / R$ is a quotient of $\ind_{\K\Z}^{\B}\sigma$, let $\Pi^+$ denote the image of $(\ind_{\K\Z}^{\B}\sigma)^+$ in $\Pi$. The space $\Pi^+$ is then a left $A\{F\}$-module, as well as a torsion $A$-module (since $\Pi$ is smooth). Recall (see \S 3 of \cite{MEcoh}) that an admissible $A$-module is an $A$-module $M$ that is torsion and such that $M^{X=0}$ is finite dimensional.

\begin{prop}\label{mecoh}
If $M$ is a finitely generated left $A\{F\}$-module that is torsion over $A$, then $M$ is admissible as an $A$-module if and only if the quotient $M/XM$ is finite dimensional over $E$.
\end{prop}

\begin{proof}
This is proposition 3.5 of \cite{MEcoh}.
\end{proof}

\begin{lemm}\label{indmodx}
The map $(\ind_{\K\Z}^{\B}\sigma)^+ \to E[F]$ given by 
\[ \sum_{\beta \in A,\delta \leq 0} \alpha(\beta,\delta) [g_{\beta,\delta}] \mapsto \sum_{n\geq 0} \left(\sum_{\beta \in A} \alpha(\beta,-n)\right) F^n\] 
gives rise to an isomorphism of $A\{F\}$-modules $(\ind_{\K\Z}^{\B}\sigma)^+ / X = E[F]$.
\end{lemm}

\begin{proof}
It is straightforward to check that the given map $(\ind_{\K\Z}^{\B}\sigma)^+ \to E[F]$ is a surjective map of $A\{F\}$-modules. Its kernel is $X \cdot (\ind_{\K\Z}^{\B}\sigma)^+$ by lemma \ref{imxkers}.
\end{proof}

\begin{lemm}\label{indpmon}
The $A\{F\}$-module $(\ind_{\K\Z}^{\B}\sigma)^+$ is generated by $[\Id]$.
\end{lemm}

\begin{proof}
The fact that if $n \geq 0$, $a,d \in \Zp^\times$ and $b \in \Zp$, then $[\smat{p^n a & b \\ 0 & d}]$ belongs to the $A\{F\}$-module generated by $[\Id]$ follows from the formula
\[ \pmat{p^n a & b \\ 0 & d} = \pmat{1&(b-p^na)d^{-1}\\0&1} \pmat{p^n&0\\0&1} \pmat{a&0\\0&d}.\]
\end{proof}

\begin{theo}\label{pipadm}
If $\Pi$ has no quotient isomorphic to $E(\sigma)$, then the $A$-module $\Pi^+$ is admissible.
\end{theo}

\begin{proof}
By proposition \ref{mecoh} above (Emerton's theorem), it is enough to show that $\Pi^+$ is finitely generated over $A\{F\}$ and that $\Pi^+ / X \Pi^+$ is a finite dimensional $E$-vector space. The finite generation follows from the fact that $\Pi^+$ is a quotient of $(\ind_{\K\Z}^{\B}\sigma)^+$, which is generated by one element over $A\{F\}$ by lemma \ref{indpmon}. 

Let $R^+=(\ind_{\K\Z}^{\B}\sigma)^+ \cap R$. We have an exact sequence of $A\{F\}$-modules
\[ R^+/ X \to (\ind_{\K\Z}^{\B}\sigma)^+ / X \to \Pi^+/X \to 0. \]
By lemma \ref{indmodx}, we have an isomorphism of $A\{F\}$-modules $(\ind_{\K\Z}^{\B}\sigma)^+ / X = E[F]$. Since any non-trivial quotient of $E[F]$ is finite dimensional over $E$, it is enough to show that $R^+$ has non-trivial image in $(\ind_{\K\Z}^{\B}\sigma)^+/X$. If this was not the case, then we would have $R^+ \subset X \cdot (\ind_{\K\Z}^{\B}\sigma)^+$. Lemma \ref{imxkers} shows that $X \cdot (\ind_{\K\Z}^{\B}\sigma)^+ \subset \ker(s)$ where $s$ is the map of definition \ref{defmaps}. If $y \in R$, then $F^n y \in R^+$ for $n \gg 0$ so that we would have $R \subset \ker(s)$ and therefore by lemma \ref{indtoe}, $\Pi$ would admit a surjective map to $E(\sigma)$.
\end{proof}

\begin{proof}[Proof of theorems A and A']
Let $\Pi$ be an infinite dimensional smooth irreducible representation of $\B$ having a central character. By theorem \ref{allquot}, we can write $\Pi = \ind_{\K\Z}^{\B}\sigma / R$ and by theorem \ref{pipadm}, $\Pi^+$ is an admissible $E\dcroc{X}$-module. Its dual $\mfont = (\Pi^+)^*$ is therefore a linearly compact topological $E$-vector space, and an $E\dcroc{X}$-module by proposition \ref{egmod}. In addition, $\mfont/X\mfont = \mfont_{\Zp}$ is finite dimensional by proposition \ref{coinv}. By theorem \ref{naklem} (Nakayama's lemma), $\mfont$ is finitely generated over $E\dcroc{X}$. 

Since $\Pi^+$ is a representation of $\B^+\Z$, its dual $\mfont$ is a representation of $\B^-\Z$. We define a $(\psi,\Gamma)$-module structure on $\mfont$ as follows: we know that it is a finitely generated module over $E\dcroc{X}$ and we set $\psi(m) = \smat{1&0\\0&p} m$ and $[a](m) = \smat{1&0\\0&a^{-1}} m$ if $a\in \Zp^\times$.

If $f : \Pi \to E$ is an element of $\Pi^*$, let $f_n$ denote the restriction of $\smat{1&0\\0&p^n}f$ to $\Pi^+$. The map $f \mapsto \{f_n\}_{n \in \ZZ}$ gives rise to an equivariant map $\Pi^* \to \projlim_\psi \mfont$. Since $\Pi^*$ is irreducible by proposition \ref{dualirred}, theorem \ref{subpsirred} applied to $\Sigma=\Pi^*$ gives us a free irreducible $(\psi,\Gamma)$-module $\nfont$ such that $\Pi^* = \projlim_\psi \nfont$. Theorem \ref{psitophi} now says that $\nfont = \mfont(\dfont)$ for some irreducible $(\phi,\Gamma)$-module $\dfont$ so that $\Pi^* = \projlim_\psi \mfont(\dfont)$. Theorem \ref{dualprodis} finally says that $\Pi = (\projlim_\psi \mfont(\dfont))^*$ which proves theorems A and A' by the bijections constructed in \S \ref{galrep} (theorem \ref{fontequiv} if $E$ is a finite field and theorem \ref{equiweil} if $E$ is algebraically closed).
\end{proof}

\begin{rema}
\label{excc}
Now that we have a classification of smooth irreducible representations of $\B$ having a central character, we can apply the same methods as in \cite{LBcc} in order to prove that in fact, every smooth irreducible representation of $\B$ over an algebraically closed field necessarily has a central character.
\end{rema}


\providecommand{\bysame}{\leavevmode ---\ }
\providecommand{\og}{``}
\providecommand{\fg}{''}
\providecommand{\smfandname}{\&}
\providecommand{\smfedsname}{\'eds.}
\providecommand{\smfedname}{\'ed.}
\providecommand{\smfmastersthesisname}{M\'emoire}
\providecommand{\smfphdthesisname}{Th\`ese}


\begin{thebibliography}{Ber10b}

\bibitem[BC08]{BCfam}
{\scshape L.~Berger {\normalfont \smfandname} P.~Colmez} -- {\og Familles de
  repr\'esentations de de {R}ham et monodromie {$p$}-adique\fg},
  \emph{Ast\'erisque} (2008), no.~319, p.~303--337, Repr{\'e}sentations
  $p$-adiques de groupes $p$-adiques. I. Repr{\'e}sentations galoisiennes et
  $(\phi,\Gamma)$-modules.

\bibitem[Ber10a]{LBmod}
{\scshape L.~Berger} -- {\og On some modular representations of the {B}orel
  subgroup of {${\mathrm{GL}}_2(\mathbf{Q}_p)$}\fg}, \emph{Compos. Math.}
  \textbf{146} (2010), no.~1, p.~58--80.

\bibitem[Ber10b]{LBgal}
\bysame , {\og Repr\'esentations modulaires de
  {${\mathrm{GL}}_2(\mathbf{Q}_p)$} et repr\'esentations galoisiennes de
  dimension 2\fg}, \emph{Ast\'erisque} (2010), no.~330, p.~263--279.

\bibitem[Ber12]{LBcc}
\bysame , {\og Central characters for smooth irreducible modular
  representations of {${\mathrm{GL}}_2(\mathbf{Q}_p)$}\fg}, \emph{Rend. Semin.
  Mat. Univ. Padova} (2012), to appear.

\bibitem[BH06]{BHgl2}
{\scshape C.~J. Bushnell {\normalfont \smfandname} G.~Henniart} -- \emph{The
  local {L}anglands conjecture for {$\mathrm{GL}(2)$}}, Grundlehren der
  Mathematischen Wissenschaften [Fundamental Principles of Mathematical
  Sciences], vol. 335, Springer-Verlag, Berlin, 2006.

\bibitem[Col10a]{PCmir}
{\scshape P.~Colmez} -- {\og {$(\phi,\Gamma)$}-modules et repr\'esentations du
  mirabolique de {${\mathrm{GL}}_2(\mathbf{Q}_p)$}\fg}, \emph{Ast\'erisque}
  (2010), no.~330, p.~61--153.

\bibitem[Col10b]{PCpgm}
\bysame , {\og Repr\'esentations de {${\mathrm{GL}}_2(\mathbf{Q}_p)$} et
  {$(\phi,\Gamma)$}-modules\fg}, \emph{Ast\'erisque} (2010), no.~330,
  p.~281--509.

\bibitem[Eme08]{MEcoh}
{\scshape M.~Emerton} -- {\og On a class of coherent rings, with applications
  to the smooth representation theory of {$\mathrm{GL}_2(\mathbf{Q}_p)$} in
  characteristic $p$\fg}, preprint, 2008.

\bibitem[Fon90]{JMFpg}
{\scshape J.-M. Fontaine} -- {\og Repr\'esentations {$p$}-adiques des corps
  locaux. {I}\fg}, in \emph{The {G}rothendieck {F}estschrift, {V}ol.\ {II}},
  Progr. Math., vol.~87, Birkh\"auser Boston, Boston, MA, 1990, p.~249--309.

\bibitem[Ked08]{KKrel}
{\scshape K.~S. Kedlaya} -- {\og Slope filtrations for relative
  {F}robenius\fg}, \emph{Ast\'erisque} (2008), no.~319, p.~259--301,
  Repr{\'e}sentations $p$-adiques de groupes $p$-adiques. I.
  Repr{\'e}sentations galoisiennes et $(\phi,\Gamma)$-modules.

\bibitem[Lef42]{SLtop}
{\scshape S.~Lefschetz} -- \emph{Algebraic {T}opology}, American Mathematical
  Society Colloquium Publications, v. 27, American Mathematical Society, New
  York, 1942.

\bibitem[Ser72]{JPSof}
{\scshape J.-P. Serre} -- {\og Propri\'et\'es galoisiennes des points d'ordre
  fini des courbes elliptiques\fg}, \emph{Invent. Math.} \textbf{15} (1972),
  no.~4, p.~259--331.

\bibitem[SV11]{SVfun}
{\scshape P.~Schneider {\normalfont \smfandname} M.-F. Vigneras} -- {\og A
  functor from smooth {$o$}-torsion representations to
  {$(\phi,\Gamma)$}-modules\fg}, in \emph{On certain {$L$}-functions}, Clay
  Math. Proc., vol.~13, Amer. Math. Soc., Providence, RI, 2011, p.~525--601.

\bibitem[Vie12a]{MVphd}
{\scshape M.~Vienney} -- \smfphdthesisname, UMPA, ENS de Lyon, 2012.

\bibitem[Vie12b]{MVart}
\bysame , {\og Repr\'esentations du sous-groupe de {B}orel de
  {${\mathrm{GL}}_2(\mathbf{Q}_p)$} en caract\'eristique $p$\fg}, preprint,
  2012.

\bibitem[Win83]{JPWcdn}
{\scshape J.-P. Wintenberger} -- {\og Le corps des normes de certaines
  extensions infinies de corps locaux; applications\fg}, \emph{Ann. Sci.
  \'Ecole Norm. Sup. (4)} \textbf{16} (1983), no.~1, p.~59--89.

\end{thebibliography}
\end{document}